\documentclass[11pt]{amsart}
\usepackage{amsmath,amssymb,hyperref,geometry,graphicx,tikz,ytableau,fp,mathdots}

\newcommand{\C}{\mathbb{C}}
\newcommand{\Z}{\mathbb{Z}}
\newcommand{\R}{\mathbb{R}}
\newcommand{\wt}{\mathrm{wt}_{\Z_3}}
\newcommand{\w}{\widetilde{\mathrm{wt}}_{\Z_3}}
\renewcommand{\geq}{\geqslant}
\renewcommand{\leq}{\leqslant}
\newtheorem{dummy}{}[section]
\newtheorem{theorem}[dummy]{Theorem}

\newtheorem{proposition}[dummy]{Proposition}
\newtheorem{conjecture}[dummy]{Conjecture}
\newtheorem{lemma}[dummy]{Lemma}
\theoremstyle{definition}
\newtheorem{example}[dummy]{Example}
\newtheorem{remark}[dummy]{Remark}
\newtheorem{definition}[dummy]{Definition}

\begin{document}

\title{Topology of $\Z_3$-equivariant Hilbert schemes}
\author[D.~Castro]{Deborah Castro}
\address{Department of Mathematics, San Francisco State University, 1600 Holloway Avenue, San Francisco, CA 94132, USA}
\email{drcastro@mail.sfsu.edu}
\author[D.~Ross]{Dustin Ross}
\address{Department of Mathematics, San Francisco State University, 1600 Holloway Avenue, San Francisco, CA 94132, USA}
\email{rossd@sfsu.edu}

\begin{abstract}
Motivated by work of Gusein-Zade, Luengo, and Melle-Hern\'andez, we study a specific generating series of arm and leg statistics on partitions, which is known to compute the Poincar\'e polynomials of $\Z_3$-equivariant Hilbert schemes of points in the plane, where $\Z_3$ acts diagonally. This generating series has a conjectural product formula, a proof of which has remained elusive over the last ten years. We introduce a new combinatorial correspondence between partitions of $n$ and $\{1,2\}$-compositions of $n$, which behaves well with respect to the statistic in question. As an application, we use this correspondence to compute the highest Betti numbers of the $\Z_3$-equivariant Hilbert schemes.
\end{abstract}

\maketitle

%\tableofcontents

\section{Introduction}

\subsection{Motivation}

Let $X$ be a smooth complex surface and let $X^{[n]}$ denote the Hilbert scheme of $n$ points on $X$; that is, the moduli space of zero-dimensional, length-$n$ subschemes of $X$. Then $X^{[n]}$ is a smooth, quasi-projective variety of complex dimension $2n$, and its topology has a rich structure. In particular, G\"{o}ttsche proved that the Poincar\'e polynomials of $X^{[n]}$ are determined from the Betti numbers of $X$ through a simple product formula \cite{Gottsche}.

In \cite{GZLMH}, Gusein-Zade, Luengo, and Melle-Hern\'andez initiated a study of the topology of \textit{equiviariant} Hilbert schemes, which can be thought of as parametrizing zero-dimensional substacks of the finite quotient stack $[X/G]$. They proved a product formula for Poincar\'e polynomials of $[\C^2/\Z_k]^{[n]}$ where the cyclic group $\Z_k$ acts anti-diagonally, but the diagonal action proved more elusive. In the case of the diagonal action of $\Z_3$, they discovered the following conjectural product formula.

\begin{conjecture}[Gusein-Zade, Luengo, and Melle-Hern\'andez \cite{GZLMH}]\label{conjecture}
Let the cyclic group $\Z_3$ act diagonally on $\C^2$ and let $b_k(-)$ denote the $k$th topological Betti number. Then
\[
1+\sum_{n> 0 \atop k\geq 0}b_{2k}\left(\left[\C^2/\Z_3 \right]^{[n]}\right)t^kq^n=\prod_{m\geq 1}\frac{1}{1-t^{m-1}q^{3m-2}}\frac{1}{1-t^{m}q^{3m-1}}\frac{1}{1-t^{m-1}q^{3m}}.
\]
\end{conjecture} 

In order to study the topology of $[\C^2/G]^{[n]}$, Gusein-Zade, Luengo, and Melle-Hern\'andez worked with a combinatorial interpretation of the Betti numbers, which we review in Subsection \ref{sec:combform} below. This combinatorial approach generalizes methods of Ellingsrud and Str{\o}mme \cite{ES} in the case of $(\C^2)^{[n]}$. Using the combinatorial formulation, Betti numbers can easily be computed for small $n$, and the computations evince the product formula in Conjecture \ref{conjecture}.

One application of the new combinatorial techniques that we develop in this paper is the verification that Conjecture \ref{conjecture} correctly computes the top Betti number of $\left[\C^2/\Z_3 \right]^{[n]}$ for all $n$, which is the content of the following theorem.

\begin{theorem}\label{maintheorem}
Let the cyclic group $\Z_3$ act diagonally on $\C^2$. Then, for all $n\geq 2$,
\[
b_{2n}\left(\left[\C^2/\Z_3 \right]^{[2n]}\right)=1, \hspace{.5cm} b_{2n}\left(\left[\C^2/\Z_3 \right]^{[2n+1]}\right)=2,
\]
and all higher Betti numbers vanish.
\end{theorem}

\begin{remark}
For $n=0$ or $1$, it is immediate from the combinatorial formulation below that $b_{0}\left(\left[\C^2/\Z_3 \right]^{[1]}\right)=b_{2}\left(\left[\C^2/\Z_3 \right]^{[2]}\right)=b_{2}\left(\left[\C^2/\Z_3 \right]^{[3]}\right)=1$, and all higher Betti numbers vanish.
\end{remark}

\subsection{Combinatorial formulation}\label{sec:combform}

Concretely, the equivariant Hilbert scheme can be defined as the following set of ideals:
\begin{align*}
\left[\C^2/\Z_3 \right]^{[n]}&=\{I\subseteq\C[x,y]:\dim_\C(\C[x,y]/I)=n \text{ and } \Z_3\cdot I=I\}\\
&=\left((\C^2)^{[n]}\right)^{\Z_3}.
\end{align*}
Here, $\Z_3=\langle\xi_3\rangle$ is generated by a primitive third root of unity, and $\xi_3\cdot(x,y)=(\xi_3 x,\xi_3 y)$. The algebraic torus $(\C^*)^2$ acts on $\left[\C^2/\Z_3 \right]^{[n]}$, and, upon choosing a general subtorus $\C^*$, the Bia{\l}ynicki-Birula decomposition allows one to compute the Betti numbers of $\left[\C^2/\Z_3 \right]^{[n]}$ from knowledge of the $\C^*$-weights on the tangent space at each $\C^*$-fixed point \cite{BB}. Since the $\C^*$-fixed points are monomial ideals, which correspond to partitions, the left-hand side of Conjecture \ref{conjecture} can be written as a sum over partitions.

To set combinatorial notation, let $\Lambda$ denote the set of partitions. For each partition $\lambda\in\Lambda$, we represent $\lambda$ as a (southwest-justified) Young diagram.  Let $|\lambda|$ denote the number of boxes in $\lambda$, and for a given box $\square\in\lambda$, let the arm $a(\square)$ and leg $l(\square)$ denote the number of boxes above and to the right of $\square$, respectively; see Figure \ref{fig:youngdiagram}. Define the $\Z_3$-weight of a partition to be
\begin{equation}\label{def:weight1}
\wt(\lambda):=\left|\{\square\in\lambda:l(\square)>1 \text{ and }a(\square)+1=l(\square) \text{ mod } 3\}\right|.
\end{equation}
A standard analysis of the tangent spaces in the Hilbert schemes, carried out in \cite{GZLMH}, shows that Conjecture \ref{conjecture} is equivalent to the following combinatorial formula.

\begin{conjecture}[Combinatorial reformulation of Conjecture \ref{conjecture}; c.f. \cite{GZLMH}]\label{conjecture2}
With $\wt(\lambda)$ defined as in \eqref{def:weight1}, 
\[
\sum_{\lambda\in\Lambda}q^{|\lambda|}t^{\wt(\lambda)}=\prod_{m\geq 1}\frac{1}{1-t^{m-1}q^{3m-2}}\frac{1}{1-t^{m}q^{3m-1}}\frac{1}{1-t^{m-1}q^{3m}}.
\]
\end{conjecture}

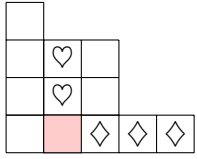
\begin{figure}
\begin{tikzpicture}[scale=.5]
    \draw[step=1cm,color=black] (0,0) grid (5,1);
    \draw[step=1cm,color=black] (0,1) grid (1,4);
    \draw[step=1cm,color=black] (1,1) grid (3,3);
    \draw[draw=black,fill=red!20] (1,0) rectangle (2,1);
    \node[] at (1.5,1.5){$\heartsuit$};
    \node[] at (1.5,2.5){$\heartsuit$};
    \node[] at (2.5,0.5){$\diamondsuit$};
    \node[] at (3.5,0.5){$\diamondsuit$};
    \node[] at (4.5,0.5){$\diamondsuit$};
\end{tikzpicture}
\caption{A partition $\lambda$ and a specified box $\square\in\lambda$ with $a(\square)=\#\heartsuit=2$, $l(\square)=\#\diamondsuit=3$, and $|\lambda|=12$.}\label{fig:youngdiagram}
\end{figure}

The condition $l(\square)>0$ in $\wt(-)$ can be cumbersome to work with, so we define a modified weight function:
\begin{equation}\label{def:weight}
\w(\lambda):=\left|\{\square\in\lambda:a(\square)+1=l(\square) \text{ mod } 3\}\right|.
\end{equation}
In Section \ref{sec:leg}, we prove that Conjecture \ref{conjecture2} is equivalent to the following formula, which is more natural from a combinatorial perspective.

\begin{conjecture}\label{combinatorialconjecture} With $\w(\lambda)$ defined as in \eqref{def:weight}, 
\[
\sum_{\lambda\in\Lambda}q^{|\lambda|}t^{\w(\lambda)}=\prod_{m\geq 1}\frac{1}{1-t^{m-1}q^{3m-2}}\frac{1}{1-t^{m}q^{3m-1}}\frac{1}{1-t^{m}q^{3m}}.
\]
\end{conjecture}

The only difference between the product formulas in Conjectures \ref{conjecture2} and \ref{combinatorialconjecture} is in the third term in the product. Theorem \ref{maintheorem} can then be reformulated combinatorially as follows.

\begin{theorem}\label{combinatorialtheorem}
If we define coefficients $\tilde b_{k,n}$ by the formula
\[
\sum_{\lambda\in\Lambda}q^{|\lambda|}t^{\w(\lambda)}=\sum_{k,n}\tilde b_{k,n}t^kq^n,
\] 
then, for $n\geq 2$,
\[
\tilde b_{n,2n}=1, \hspace{.5cm} \tilde b_{n,2n+1}=3,
\]
and $\tilde b_{k,n}=0$ for all $k>n/2$.
\end{theorem}

\subsection{Methods}

In Proposition \ref{prop:decompose}, we show that every partition of $n$ can be decomposed uniquely into a $\{1,2\}$-composition of $n$; that is, a sequence $(a_1,\dots,a_k)$ where $a_i\in\{1,2\}$ and $\sum_{i=1}^k a_i=n$. Given such a composition corresponding to a partition $\lambda$, we then prove in Theorem \ref{thm:compare} that $\w(\lambda)$ is equal to the number of times $2$ appears. Therefore, Conjecture \ref{combinatorialconjecture} boils down to understanding exactly which compositions arise from partitions. In general, this seems to be a very difficult question: partition numbers grow at a much slower and more mysterious rate than the number of $\{1,2\}$-compositions, which are counted by Fibonacci numbers. In Proposition \ref{prop:admissible}, we explicitly describe the $\{1,2\}$-compositions that arise when there is at most one occurrence of $1$ in the composition, and Theorem \ref{combinatorialtheorem} follows.

\subsection{Generalizations and relation to work of others}

In a forthcoming paper, Paul Johnson has generalized the product formula in Conjecture \ref{conjecture} to quotients of $\C^2$ by any finite abelian group \cite{Johnson}. He has also proved that the product formula holds asymptotically as $q\rightarrow\infty$. Since his asymptotic results compute the small (relative to $n$) Betti numbers, our methods, in the $\Z_3$ case, provide evidence for Conjecture \ref{conjecture} that is orthogonal to his.

Our original goal was to devise methods that would be applicable to more general actions of finite abelian groups. A natural generalization of the situation studied herein is the diagonal action of $\Z_k$ on $\C^2$; combinatorially, this amounts to replacing the mod $3$ conditions with mod $k$ conditions. In that generality, Proposition \ref{prop:leg} and the second part of Theorem \ref{thm:compare} readily generalize, as the reader may check; however, the first part of Theorem \ref{thm:compare} is special to the $\Z_3$ case.

\subsection{Acknowledgements}

The authors are grateful to Federico Ardila, Matthias Beck, and Emily Clader for their interest in this work and for comments on early drafts of this manuscript, as well as the Department of Mathematics at San Francisco State University for providing a supportive and encouraging atmosphere for working on this project. They are also indebted to Paul Johnson for instructive conversations regarding orbifold Hilbert schemes, and for sharing his draft \cite{Johnson}.

\section{Resolving the leg condition}\label{sec:leg}

The following theorem allows us to work with the modified weight $\w(\lambda)$, which is more conducive to our combinatorial methods.

\begin{proposition}\label{prop:leg}
With $\wt(\lambda)$ and $\w(\lambda)$ defined as in \eqref{def:weight1} and \eqref{def:weight}, respectively, we have
\[
\sum_{\lambda\in\Lambda}q^{|\lambda|}t^{\w(\lambda)}=\left(\prod_{m\geq 1} \frac{1-t^{m-1}q^{3m}}{1-t^{m}q^{3m}}\right)\sum_{\lambda\in\Lambda}q^{|\lambda|}t^{\wt(\lambda)}
\]
\end{proposition}

\begin{proof}
The difference between $\wt(\lambda)$ and $\w(\lambda)$ only concerns boxes $\square\in\lambda$ where $l(\square)=0$ and $a(\square)+1=0$ mod $3$. Let $\Lambda'$ be the set of partitions that do not contain any such boxes. In other words, $\lambda\in\Lambda'$ if and only if the boundary of $\lambda$ does not contain more than 2 consecutive downward steps. By definition,
\[
\sum_{\lambda'\in\Lambda'}q^{|\lambda'|}t^{\w(\lambda')}=\sum_{\lambda'\in\Lambda'}q^{|\lambda'|}t^{\wt(\lambda')}.
\]

Let $\Lambda''$ denote the set of partitions where the size of each column is a multiple of three. There is a bijection
\[
f:\Lambda'\times\Lambda''\rightarrow\Lambda
\]
defined by inserting the rows of $\lambda''\in\Lambda''$ into $\lambda'\in\Lambda'$ at the maximum height so that the result is still a Young diagram. See Figure \ref{fig:bijection}.

\begin{figure}[h]
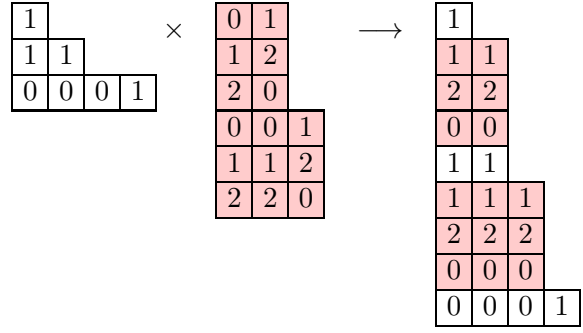

\[
\ytableausetup{boxsize=1.2em}
\begin{ytableau}
 1 \\
 1 &  1 \\
 0 &  0 & 0 & 1 
\end{ytableau}
\times\;\;\;
\begin{ytableau}
*(red!20) 0 & *(red!20) 1\\
*(red!20) 1 & *(red!20) 2\\
*(red!20) 2 & *(red!20) 0 \\
*(red!20) 0 & *(red!20) 0 & *(red!20) 1\\ 
*(red!20) 1 & *(red!20) 1 & *(red!20) 2\\ 
*(red!20) 2 & *(red!20) 2 & *(red!20) 0
\end{ytableau}
\;\;\;\longrightarrow\;\;\;
\begin{ytableau}
 1 \\
*(red!20) 1 & *(red!20) 1\\
*(red!20) 2 & *(red!20) 2\\
*(red!20) 0 & *(red!20) 0 \\
 1 &  1 \\
*(red!20) 1 & *(red!20) 1 & *(red!20) 1\\ 
*(red!20) 2 & *(red!20) 2 & *(red!20) 2\\ 
*(red!20) 0 & *(red!20) 0 & *(red!20) 0\\
 0 &  0 & 0 & 1 
\end{ytableau}
\]
\caption{An example of the map $f$. All boxes are labeled with the statistic $a(\square)+1-l(\square)$ mod $3$, and $\w(-)$ counts the number of zeros appearing in each diagram.}\label{fig:bijection}
\end{figure}

Each $\square\in\lambda'$ has a corresponding $\square\in f(\lambda',\lambda'')$, and the statistic $a(\square)+1-l(\square)$ mod $3$ is the same for both of these boxes because $f$ leaves the leg unchanged and only alters the arm by a multiple of three. Similarly, for each block of rows in $\lambda''$ that have the same length, there is a corresponding block of rows in $f(\lambda',\lambda'')$. The map $f$ leaves the legs of the boxes in these blocks unchanged, but it can change the arms, which alters the statistic by adding a constant to each column of the block (mod $3$). Since each column in each block has the same number of zeros, ones, and twos, this simply results in a permutation of the statistic in each column, which preserves the number of zeros. Thus, it follows that
\[
\wt(f(\lambda',\lambda''))=\wt(\lambda')+\wt(\lambda'') \;\;\;\text{ and }\;\;\; \w(f(\lambda',\lambda''))=\w(\lambda')+\w(\lambda'').
\]

Since
\[
\sum_{\lambda''\in\Lambda''}q^{|\lambda''|}t^{\w(\lambda'')}=\prod_{m\geq 1} \frac{1}{1-t^{m}q^{3m}}\;\;\;\text{ and }\;\;\;\sum_{\lambda''\in\Lambda''}q^{|\lambda''|}t^{\wt(\lambda'')}=\prod_{m\geq 1} \frac{1}{1-t^{m-1}q^{3m}},
\]
it follows that
\begin{align*}
\sum_{\lambda}q^{|\lambda|}t^{\w(\lambda)}&=\left(\sum_{\lambda'\in\Lambda'}q^{|\lambda'|}t^{\w(\lambda')}\right)\left(\sum_{\lambda''\in\Lambda''}q^{|\lambda''|}t^{\w(\lambda'')}\right)\\
&=\left(\sum_{\lambda'\in\Lambda'}q^{|\lambda'|}t^{\wt(\lambda')}\right)\left(\sum_{\lambda''\in\Lambda''}q^{|\lambda''|}t^{\wt(\lambda'')}\right)\left(\prod_{m\geq 1} \frac{1-t^{m-1}q^{3m}}{1-t^{m}q^{3m}}\right)\\
&=\left(\prod_{m\geq 1} \frac{1-t^{m-1}q^{3m}}{1-t^{m}q^{3m}}\right)\sum_{\lambda}q^{|\lambda|}t^{\wt(\lambda)}.
\end{align*}

\end{proof}

As a corollary, we see that Conjecture \ref{conjecture2} is equivalent to Conjecture \ref{combinatorialconjecture}, and Theorem \ref{maintheorem} is equivalent to Theorem \ref{combinatorialtheorem}.

\section{Partitions and compositions}

\subsection{Dyson maps}
We start by defining two maps on partitions, $\rho_1$ and $\psi_2$, that we will use throughout the rest of this paper. In \cite{Pak}, these maps are referred to as \emph{Dyson maps}.

Let $\lambda$ be a partition, represented as a Young diagram. The map $\rho_1$ removes the first row of $\lambda$, adds one box to it, then inserts these boxes as a new first column. See Figure \ref{fig:fig3.1}. We must take into consideration that the new column must be at least as big as what remains of the original first column. If $j$ is the number of boxes in the first row of $\lambda$ and $k$ is the number of boxes in the first column, this is equivalent to requiring that $j \geq k - 2$.

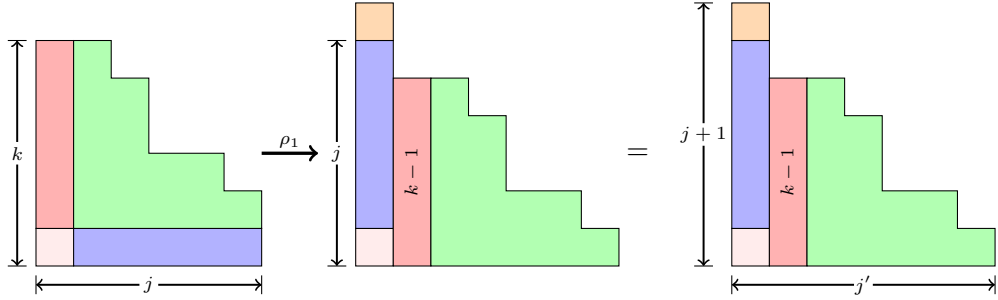
\begin{figure}[h]
    \centering
    \begin{tikzpicture}[scale=0.5]
    % FIRST YOUNG DIAGRAM
    \draw[fill=pink!30] (0,0) rectangle (1,1);
    \draw[fill=blue!30] (1,0) rectangle (6,1);
    \draw[fill=red!30] (0,1) rectangle (1,6);
    \draw[fill=green!30] (1,1) -- (6,1) -- (6,2) -- (5,2) -- (5,3) -- (3,3) -- (3,5) -- (2,5) -- (2,6)-- (1,6) -- cycle;
    
    % HORIZONTAL DELIMITS
    \draw (0,-0.25) -- (0,-0.75);
    \draw (6,-0.25) -- (6,-0.75);
    \draw[<-,line width=0.75pt] (0,-0.5) -- (2.75,-0.5);
    \draw[->,line width=0.75pt] (3.25,-0.5) -- (6,-0.5);
    \node at (3,-0.5) {\tiny $j$};
    
    % VERTICAL DELIMITS
    \draw (-0.75,0) -- (-0.25,0);
    \draw (-0.75,6) -- (-0.25,6);
    \draw[->,line width=0.75pt] (-0.5,3.3) -- (-0.5,6);
    \draw[->,line width=0.75pt] (-0.5,2.7) -- (-0.5,0);
    \node at (-0.5,3) {\tiny $k$}; 
    
    % ARROW
    \draw[->,line width=1.25pt] (6,3) -- (7.5,3);
    \node at (6.75,3.3) {\tiny $\rho_1$};
    
    % SECOND YOUNG DIAGRAM
    \draw[fill=pink!30] (8.5,0) rectangle (9.5,1);
    \draw[fill=blue!30] (8.5,1) rectangle (9.5,6);
    \draw[fill=orange!30] (8.5,6) rectangle (9.5,7);
    \draw[fill=red!30] (9.5,0) rectangle (10.5,5);
    
    % POLYGON
    \draw[fill=green!30] (10.5,0) -- (15.5,0) -- (15.5,1) -- (14.5,1) -- (14.5,2) -- (12.5,2) -- (12.5,4) -- (11.5,4) -- (11.5,5) -- (10.5,5) -- cycle;
    \node[rotate=90] at (10,2.5) {\tiny $k-1$};
    
    % VERTICAL DELIMITS
    \draw (7.75,0) -- (8.25,0);
    \draw (7.75,6) -- (8.25,6);
    \draw[->,line width=0.75pt] (8,3.3) -- (8,6);
    \draw[->,line width=0.75pt] (8,2.7) -- (8,0);
    \node at (8,3) {\tiny $j$};
    
    % EQUAL SIGN BETWEEN DIAGRAMS
    \node at (16,3) {=};
    
    % THIRD YOUNG DIAGRAM
    \draw[fill=pink!30] (18.5,0) rectangle (19.5,1);
    \draw[fill=blue!30] (18.5,1) rectangle (19.5,6);
    \draw[fill=orange!30] (18.5,6) rectangle (19.5,7);
    \draw[fill=red!30] (19.5,0) rectangle (20.5,5);
    
    % POLYGON
    \draw[fill=green!30] (20.5,0) -- (25.5,0) -- (25.5,1) -- (24.5,1) -- (24.5,2) -- (22.5,2) -- (22.5,4) -- (21.5,4) -- (21.5,5) -- (20.5,5) -- cycle;
    \node[rotate=90] at (20,2.5) {\tiny $k-1$};
    
    % VERTICAL DELIMITS
    \draw (17.5,0) -- (18,0);
    \draw (17.5,7) -- (18,7);
    \draw[->,line width=0.75pt] (17.75,3.2) -- (17.75,0);
    \draw[->,line width=0.75pt] (17.75,3.8) -- (17.75,7);
    \node at (17.75,3.5) {\tiny $j+1$};
    
    % HORIZONTAL DELIMITS
    \draw (18.5,-0.25) -- (18.5,-0.75);
    \draw (25.5,-0.25) -- (25.5,-0.75);
    \draw[<-,line width=0.75pt] (18.5,-0.5) -- (21.75,-0.5);
    \draw[->,line width=0.75pt] (22.25,-0.5) -- (25.5,-0.5);
    \node at (22,-0.5) {\tiny $j'$};
\end{tikzpicture}
    \caption{The map $\rho_1$ applied to $\lambda$, defined whenever $j \geq k - 2$.}
    \label{fig:fig3.1}
\end{figure}

Similarly, let $\psi_2$ be the map that removes the first column of $\lambda$, adds two boxes to it, then inserts these boxes as the new first row. See Figure \ref{fig:fig3.2}. We must take into consideration that the new row must be at least as big as what remains of the original first row, which is equivalent to the condition that $j \leqslant k + 3$. 

\begin{figure}[h]
    \centering
    \begin{tikzpicture}[scale=0.5]
    % FIRST YOUNG DIAGRAM
    \draw[fill=pink!30] (0,0) rectangle (1,1);
    \draw[fill=blue!30] (1,0) rectangle (6,1);
    \draw[fill=red!30] (0,1) rectangle (1,6);
    \draw[fill=green!30] (1,1) -- (6,1) -- (6,2) -- (5,2) -- (5,3) -- (3,3) -- (3,5) -- (2,5) -- (2,6)-- (1,6) -- cycle;
    
    % HORIZONTAL DELIMITS
    \draw (0,-0.25) -- (0,-0.75);
    \draw (6,-0.25) -- (6,-0.75);
    \draw[<-,line width=0.75pt] (0,-0.5) -- (2.75,-0.5);
    \draw[->,line width=0.75pt] (3.25,-0.5) -- (6,-0.5);
    \node at (3,-0.5) {\tiny $j$};
    
    % VERTICAL DELIMITS
    \draw (-0.75,0) -- (-0.25,0);
    \draw (-0.75,6) -- (-0.25,6);
    \draw[->,line width=0.75pt] (-0.5,3.3) -- (-0.5,6);
    \draw[->,line width=0.75pt] (-0.5,2.7) -- (-0.5,0);
    \node at (-0.5,3) {\tiny $k$};
    
    % ARROW BETWEEN DIAGRAMS
    \draw[->,line width=1.25pt] (6,3) -- (7.5,3);
    \node at (6.75,3.3) {\tiny $\psi_2$};
    
    % SECOND YOUNG DIAGRAM
    \draw[fill=pink!30] (8,0) rectangle (9,1);
    \draw[fill=red!30] (9,0) rectangle (14,1);
    \draw[fill=orange!30] (14,0) rectangle (15,1);
    \draw[fill=orange!30] (15,0) rectangle (16,1);
    \draw[fill=blue!30] (8,1) rectangle (13,2);
    
    % POLYGON
    \draw[fill=green!30] (8,2) -- (13,2) -- (13,3) -- (12,3) -- (12,4) -- (10,4) -- (10,6) -- (9,6) -- (9,7) -- (8,7) -- cycle;
    
    % HORIZONTAL DELIMITS
    \draw (8,-0.25) -- (8,-0.75);
    \draw (14,-0.25) -- (14,-0.75);
    \draw[<-,line width=0.75pt] (8,-0.5) -- (10.75,-0.5);
    \draw[->,line width=0.75pt] (11.25,-0.5) -- (14,-0.5);
    \node at (11,-0.5) {\tiny $k$};
    
    \node at (10.5,1.5) {\tiny $j-1$};
    
    % EQUAL SIGN BETWEEN DIAGRAMS
    \node[] at (16,3) {=};
    
    % THIRD YOUNG DIAGRAM
    \draw[fill=pink!30] (18,0) rectangle (19,1);
    \draw[fill=red!30] (19,0) rectangle (24,1);
    \draw[fill=orange!30] (24,0) rectangle (25,1);
    \draw[fill=orange!30] (25,0) rectangle (26,1);
    \draw[fill=blue!30] (18,1) rectangle (23,2);
    
    % POLYGON
    \draw[fill=green!30] (18,2) -- (23,2) -- (23,3) -- (22,3) -- (22,4) -- (20,4) -- (20,6) -- (19,6) -- (19,7) -- (18,7) -- cycle;
    
    % HORIZONTAL DELIMITS    
    \draw (18,-0.25) -- (18,-0.75);
    \draw (26,-0.25) -- (26,-0.75);
    \draw[<-,line width=0.75pt] (18,-0.5) -- (21.3,-0.5);
    \draw[->,line width=0.75pt] (22.7,-0.5) -- (26,-0.5);
    \node at (22,-0.5) {\tiny $k+2$};
    
    \node at (20.5,1.5) {\tiny $j-1$};
    
    % VERTICAL DELIMITS
    \draw (17.25,0) -- (17.75,0);
    \draw (17.25,7) -- (17.75,7);
    \draw[<-,line width=0.75pt] (17.5,7) -- (17.5,3.8);
    \draw[->,line width=0.75pt] (17.5,3.2) -- (17.5,0);
    \node at (17.5,3.5) {\tiny $k'$};
\end{tikzpicture}
    \caption{The map $\psi_2$ applied to $\lambda$, defined whenever $j \leqslant k + 3$.}
    \label{fig:fig3.2}
\end{figure}
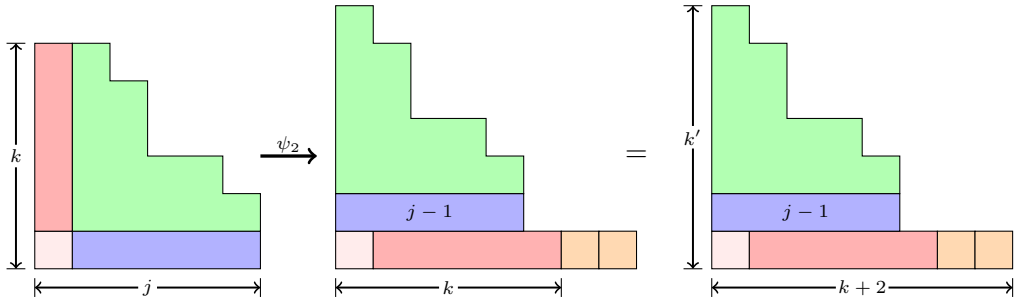

\subsection{Decomposing partitions}

Since $\rho_1$ increases the size of a partition by $1$, and $\rho_2$ increases the size of a partition by $2$, the following result relates partitions of size $n$ to $\{1,2\}$-composition of $n$.

\begin{proposition}\label{prop:decompose}
Every partition can be written uniquely as a sequence of $\rho_1$s and $\psi_2$s applied to the empty partition.
\end{proposition}

\begin{proof}
Notice that $\rho_1$ and $\psi_2$ have natural inverses $\rho_1^{-1}$ and $\psi_2^{-1}$. The map $\rho_1^{-1}$ removes the first column, takes away one box, and inserts it as a new first row, while the map $\psi_2^{-1}$ removes the first row, takes away two boxes, and inserts it as a new first column. We require that the outcomes are Young diagrams, which translates to $\rho_1^{-1}$  only being well-defined when $j\leq k$, and $\psi_2^{-1}$ only being well-defined when $j>k$. Thus, one and only one of these maps is well-defined, and for each $\lambda$, there is a unique sequence of $\rho_1^{-1}$s and $\psi_2^{-1}$s that, when applied to $\lambda$, yield the empty diagram.
\end{proof}

Let $P_n$ be the set of partitions of size $n$ and $C_n^{\{1,2\}}$ the set of $\{1,2\}$-compositions of $n$. Proposition \ref{prop:decompose} allows us to define an injective map
\begin{align*}
\phi:P_n&\rightarrow C_n^{\{1,2\}},
\end{align*}
More precisely, after writing $\lambda$ as a sequence of $\rho_1$s and $\psi_2$s applied to the empty diagram, $\phi(\lambda)$ is the corresponding sequence of subscripts. This map is far from surjective. In fact, the size of $C_n^{\{1,2\}}$ is the $(n+1)$st Fibonacci number, which is known to grow at a much greater rate than the size of $P_n$. 

From our perspective, the important compositions are those that lie in the image of $\phi$.

\begin{definition}
We say that a $\{1,2\}$-composition of $n$ is \emph{admissible} if it is in the image of $\phi$.
\end{definition}

The following example exhibits a subsequence that will never appear in an admissible composition.

\begin{example}\label{example:unadmissible}
For any partition $\lambda$, $\psi_2\psi_2\rho_1\psi_2\lambda$ is not a well-defined Young diagram; see Figure \ref{fig:fig3.19}. Thus, any $\{1,2\}$-composition containing the subsequence $(2,2,1,2)$ is not admissible.

 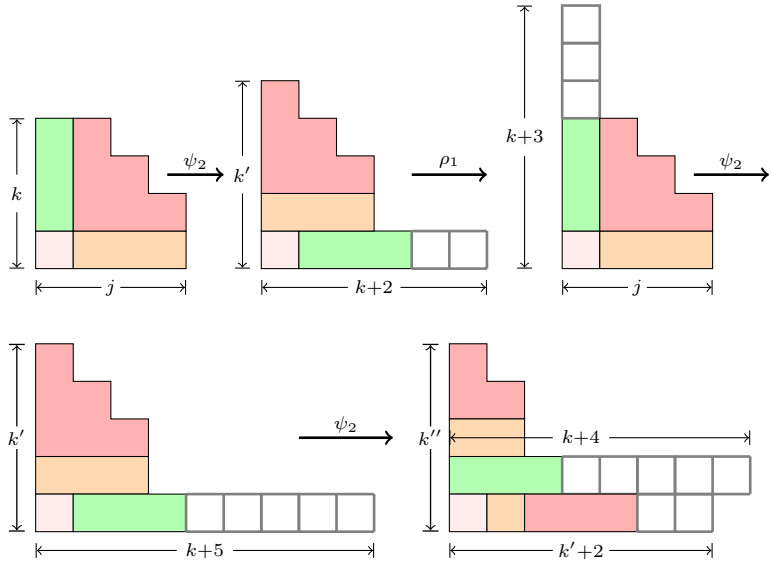
\begin{figure}[h]
        \centering
\begin{tikzpicture}[scale=0.5]
    % FIGURE 1
    \draw[fill=pink!30] (0,7) rectangle (1,8);
    \draw[fill=orange!30] (1,7) rectangle (4,8);
    \draw[fill=green!30] (0,8) rectangle (1,11);
    
    % POLYGON
    \draw[fill=red!30] (1,8) -- (4,8) -- (4,9) -- (3,9) -- (3,10) -- (2,10) -- (2,11) -- (1,11) -- cycle;
    
    % Horizontal Delimits
    \draw[|<-] (0,6.5) -- (1.75,6.5);
    \draw[->|] (2.25,6.5) -- (4,6.5);
    \node[] at (2,6.5) {\tiny $j$};
    
    % Vertical Delimits
    \draw[|<-] (-0.5,11) -- (-0.5,9.5);
    \draw[->|] (-0.5,8.5) -- (-0.5,7);
    \node[] at (-0.5,9) {\tiny $k$};
    
    %------------------------------------------------------------------------------------------
    \draw[->,line width=1pt] (3.5,9.5) -- (5,9.5);
    \node at (4.25,9.85) {\tiny ${\psi_2}$};
    %------------------------------------------------------------------------------------------
    
    % FIGURE 2
    \draw[fill=pink!30] (6,7) rectangle (7,8);
    \draw[fill=green!30] (7,7) rectangle (10,8);
    \draw[color=gray,line width=1pt] (10,7) grid (12,8);
    \draw[fill=orange!30] (6,8) rectangle (9,9);
    
    % POLYGON
    \draw[fill=red!30] (6,9) -- (9,9) -- (9,10) -- (8,10) -- (8,11) -- (7,11) -- (7,12) -- (6,12) -- cycle;
    
    % Horizontal Delimits
    \draw[|<-] (6,6.5) -- (8.3,6.5);
    \draw[->|] (9.7,6.5) -- (12,6.5);
    \node[] at (9,6.5) {\tiny $k{+}2$};
    %\node[] at (7.5,8.5) {\tiny $j{-}1$};
    
    % Vertical Delimits
    \draw[|<-] (5.5,12) -- (5.5,10);
    \draw[->|] (5.5,9) -- (5.5,7);
    \node[] at (5.5,9.5) {\tiny $k'$};
    
    %------------------------------------------------------------------------------------------
    \draw[->,line width=1pt] (10,9.5) -- (12,9.5);
    \node at (11,9.85) {\tiny ${\rho_1}$};
    %------------------------------------------------------------------------------------------
    
    % FIGURE 3
    \draw[fill=pink!30] (14,7) rectangle (15,8);
    \draw[fill=orange!30] (15,7) rectangle (18,8);
    \draw[fill=green!30] (14,8) rectangle (15,11);
    \draw[color=gray,line width=1pt] (14,11) grid (15,13);
    \draw[color=gray,line width=1pt] (14,13) rectangle (15,14);
    
    % POLYGON
    \draw[fill=red!30] (15,8) -- (18,8) -- (18,9) -- (17,9) -- (17,10) -- (16,10) -- (16,11) -- (15,11) -- cycle;
    
    % Vertical Delimits
    \draw[|<-] (13,14) -- (13,11);
    \draw[->|] (13,10) -- (13,7);
    \node[] at (13,10.5) {\tiny $k{+}3$};
    
    % Horizontal Delimits
    \draw[|<-] (14,6.5) -- (15.75,6.5);
    \draw[->|] (16.25,6.5) -- (18,6.5);
    \node[] at (16,6.5) {\tiny $j$};
    
    %------------------------------------------------------------------------------------------
    \draw[->,line width=1pt] (17.5,9.5) -- (19.5,9.5);
    \node at (18.5,9.85) {\tiny ${\psi_2}$};
    %------------------------------------------------------------------------------------------

    % FIGURE 4
    % Polygon
    \draw[fill=red!30] (0,2) -- (3,2) -- (3,3) -- (2,3) -- (2,4) -- (1,4) -- (1,5) -- (0,5) --cycle;
    
    % Row 2
    \draw[fill=orange!30] (0,1) rectangle (3,2);
    
    % Row 1
    \draw[fill=pink!30] (0,0) rectangle (1,1);
    \draw[fill=green!30] (1,0) rectangle (4,1);
    \draw[color=gray,line width=1pt] (4,0) grid (6,1);
    \draw[color=gray,line width=1pt] (6,0) rectangle (7,1);
    \draw[color=gray,line width=1pt] (7,0) grid (9,1);
    
    % Vertical Delimits
    \draw[|<-,line width=0.5pt] (-0.5,5) -- (-0.5,2.8);
    \draw[->|,line width=0.5pt] (-0.5,2.2) -- (-0.5,0);
    \node at (-0.5,2.5) {\tiny $k'$};
    
    % Horizontal Delimits
    \draw[|<-] (0,-0.5) -- (3.75,-0.5);
    \draw[->|] (5.25,-0.5) -- (9,-0.5);
    \node at (4.5,-0.5) {\tiny $k{+}5$};
    
    % Inside Labels
    %\node at (1.5,1.5) {\tiny $j{-}1$};
    
    % Dotted line
    %\draw[color=brown,dashed,line width=2pt] (1,5.5) -- (1,-0.5);
    
    %------------------------------------------------------------------------------------------
    \draw[->,line width=1pt] (7,2.5) -- (9.5,2.5);
    \node at (8.25,2.85) {\tiny ${\psi_2}$};
    %------------------------------------------------------------------------------------------
    
    % FIGURE 5
    % Polygon
    \draw[fill=red!30] (11,3) -- (13,3) -- (13,4) -- (12,4) -- (12,5) -- (11,5) -- cycle;
    
    % Row 3
    \draw[fill=orange!30] (11,2) rectangle (13,3);
    
    % Row 2
    \draw[fill=green!30] (11,1) rectangle (14,2);
    \draw[color=gray,line width=1pt] (14,1) grid (16,2);
    \draw[color=gray,line width=1pt] (16,1) rectangle (17,2);
    \draw[color=gray,line width=1pt] (17,1) grid (19,2);
    
    % Row 1
    \draw[fill=pink!30] (11,0) rectangle (12,1);
    \draw[fill=orange!30] (12,0) rectangle (13,1);
    \draw[fill=red!30] (13,0) rectangle (16,1);
    \draw[color=gray,line width=1pt] (16,0) grid (18,1);

    % Vertical Delimits
    \draw[|<-,line width=0.5pt] (10.5,5) -- (10.5,2.8);
    \draw[->|,line width=0.5pt] (10.5,2.2) -- (10.5,0);
    \node at (10.5,2.5) {\tiny $k''$};
    
    % Horizontal Delimits
    \draw[|<-] (11,-0.5) -- (13.75,-0.5);
    \draw[->|] (15.25,-0.5) -- (18,-0.5);
    \node at (14.5,-0.5) {\tiny $k'{+}2$};
    
     \draw[|<-] (11,2.5) -- (13.75,2.5);
    \draw[->|] (15.25,2.5) -- (19,2.5);
    \node at (14.5,2.5) {\tiny $k{+}4$};
    
    % Inside Labels
    %\node at (14.5,1.5) {\tiny $k{+}4$};
    %\node at (12,2.5) {\tiny $j{-}2$};
\end{tikzpicture}
\caption{The sequence $\psi_2 \psi_2 \rho_1 \psi_2$ is undefined for any $\lambda$ because $k' \leqslant k + 1$.}\label{fig:fig3.19}
\end{figure}
\end{example}

\subsection{The effect of $\rho_1$ and $\psi_2$ on $\w(-)$}

We now study how $\w(-)$ behaves with respect to the decomposition in Proposition \ref{prop:decompose}. In order to do so, we introduce a few preliminary conventions and definitions. Start by placing each Young diagram in the first quadrant of $\R^2$ so that the lower left corner is at the origin and each box is a unit square.

\begin{definition}
The \emph{boundary sequence} of $\lambda$ is the sequence along the northeast boundary of $\lambda$ that is induced by the labeling in Figure \ref{fig:fig3.4}.
\end{definition}

\begin{figure}[h]
    \centering
    \begin{tikzpicture}[scale=0.75]
    % GRID    
    \draw[step=1cm,gray!50] (0,0) grid (8,6);
    
    % AXES
    \draw[->,draw=black,line width=0.5pt] (0,-0.25) -- (0,7) node[anchor=north east] {\small };
    \draw[->,draw=black,line width=0.5pt] (-0.25,0) -- (9,0) node[anchor=west] {\small };
    %\node at (-0.125,-0.125) {\tiny $0$};
    
    % NUMBER LABELS
    \foreach \x in {0,1,2,3,4,5,6,7} \foreach \y in {0,1,2,3,4,5,6}
        \draw[color=black] node at (\x+.5,\y) {\scriptsize \FPeval{\result}{clip(\y+\x+1)}\result};
        
    \foreach \x in {0,1,2,3,4,5,6,7,8} \foreach \y in {0,1,2,3,4,5}
        \draw[color=black] node at (\x,\y+.5) {\scriptsize \FPeval{\result}{clip(\y+\x)}\result};
\end{tikzpicture}
    \caption{Line segments in the first quadrant.}
    \label{fig:fig3.4}
\end{figure}
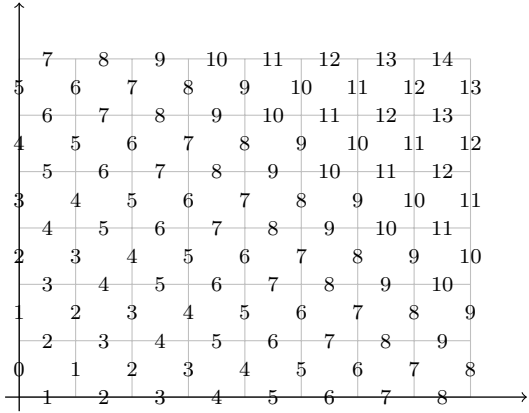

The following result shows how the boundary sequence is related to $\w(-)$.

\begin{lemma} \label{lem:boundary}
Let $\square \in \lambda$ and let $i$ be the label in the boundary sequence directly above $\square$ and $j$ the label in the boundary sequence directly to the right of $\square$. Then $\square$ contributes to $\w(\lambda)$ (i.e. $a(\square)+1=l(\square)$ mod $3$) if and only if $i = j\ \rm{mod}\ 3$.
\end{lemma}

\begin{proof}
In terms of $i$ and $j$, $a(\square)$ is the number of vertical line segments in the boundary between $i$ and $j$, and $l(\square)$ is the number of horizontal line segments in the boundary between $i$ and $j$. Since $i$ and $j$ are part of the boundary sequence, then taking a walk along the boundary gives a relation: 
\[
i \pm 1 \pm 1 \cdots \pm 1\pm 1 - 1 = j,
\]
where each $-1$ corresponds to a vertical line segment and each $+1$ corresponds to a horizontal line segment in the boundary of $\lambda$. Thus, we obtain $i - a(\square) + l(\square) - 1 = j$, implying that $i - j = a(\square) - l(\square) + 1$. 
\end{proof}

\begin{example} \label{ex:boundary}
Let $\lambda$ be the Young diagram in Figure \ref{fig:fig3.6}. The boundary sequence is $(3,4,5,6,5,6,7,6)$, and $\w(\lambda) = 4$. 
    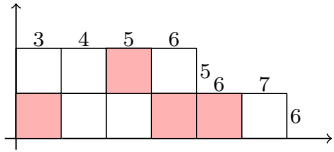
\begin{figure}[h]
        \centering
      \begin{tikzpicture}[scale=0.6]
     % AXES
    \draw[->,draw=black,line width=0.5pt] (0,-0.25) -- (0,3) node[anchor=north east] {\small};
    \draw[->,draw=black,line width=0.5pt] (-0.25,0) -- (7,0) node[anchor=west] {\small };
    
    % POLYGON
   % \draw[draw=red,line width=1.2pt] (0,0) -- (6,0) -- (6,1) -- (4 ,1) -- (4,2) -- (0,2) -- cycle;
    
    % GRIDS
    \draw[step=1cm] (0,0) grid (6,1);
    \draw[step=1cm] (0,1) grid (4,2);
    
    % NUMBER NODES
    \node[] at (0.5,2.2) {\scriptsize $3$};
    \node[] at (1.5,2.2) {\scriptsize $4$};
    \node[] at (2.5,2.2) {\scriptsize $5$};
    \node[] at (3.5,2.2) {\scriptsize $6$};
    \node[] at (4.2,1.5) {\scriptsize $5$};
    \node[] at (4.5,1.2) {\scriptsize $6$};
    \node[] at (5.5,1.2) {\scriptsize $7$};
    \node[] at (6.2,0.5) {\scriptsize $6$};
    
    % X NODES
    \draw[fill=red!30] (2,1) rectangle (3,2);
    \draw[fill=red!30] (0,0) rectangle (1,1);
    \draw[fill=red!30] (3,0) rectangle (4,1);
    \draw[fill=red!30] (4,0) rectangle (5,1);
\end{tikzpicture}
        \caption{Boundary sequence of $\lambda$. The four shaded boxes contribute to $\w(\lambda)$.}
        \label{fig:fig3.6}
    \end{figure}
\end{example}

We now come to the main result of this section.

\begin{theorem}\label{thm:compare}
If $\rho_1\lambda$ is defined, then 
\[
\w(\rho_1\lambda)=\w(\lambda).
\] 
If $\psi_2\lambda$ is defined, then 
\[
\w(\psi_2\lambda)=\w(\lambda)+1.
\]
\end{theorem}

\begin{proof}
In order to compare $\w(\lambda)$ with $\w(\rho_1\lambda)$, it is useful to overlap $\lambda$ and $\rho_1\lambda$ in the same diagram. To do so, shift $\lambda$ to the right by one unit in $\R^2$ and shift $\rho_1\lambda$ up by one unit. See Figure \ref{fig:fig3.10}, for example.

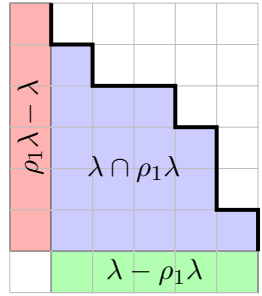
\begin{figure}[h]
        \centering
        \begin{tikzpicture}[scale=0.55]
    % RECTANGLES
    \draw[fill=red!30] (0,1) rectangle (1,7);
    \draw[fill=green!30] (1,0) rectangle (6,1);

    % GRID    

    % DIAGONALS
    %\draw[color=pink,line width=1pt] (1,7.5) -- (6.5,2);
    
    % POLYGON
    \draw[fill=blue!20] (1,1) -- (6,1) -- (6,2) -- (5,2) -- (5,4) -- (4,4)-- (4,5) -- (2,5) -- (2,6) -- (1,6) -- cycle;
    \draw[step=1cm,gray!50] (0,0) grid (6,7);
    \draw[line width = 1.5pt] (6,1) -- (6,2) -- (5,2) -- (5,4) -- (4,4)-- (4,5) -- (2,5) -- (2,6) -- (1,6) -- (1,7);
    
    \node[] at (3.5,0.5) {$\lambda - \rho_1 \lambda$};
    \node[] at (3,3) {$\lambda\cap\rho_1 \lambda$};
    \node[rotate=90] at (0.5,4) {$\rho_1 \lambda - \lambda$};
    
    %\foreach \x in {0,1,2,3,4,5} \foreach \y in {0,1,2,3,4,5,6,7}
        %\draw[color=black] node at (\x+.5,\y) {\scriptsize \FPeval{\result}{clip(\y+\x+1)}\result};
        
    %\foreach \x in {0,1,2,3,4,5,6} \foreach \y in {0,1,2,3,4,5,6}
        %\draw[color=black] node at (\x,\y+.5) {\scriptsize \FPeval{\result}{clip(\y+\x)}\result};
\end{tikzpicture}
        \caption{An example of $\lambda$ and $\rho_1\lambda$ overlapping in first quadrant. The bold path is the \emph{boundary} of $\lambda\cap\rho_1\lambda$. }
        \label{fig:fig3.10}
\end{figure}

The weight of $\lambda\cap\rho_1\lambda$ is independent of which diagram we consider, so we need only compare the weight of $\lambda-\rho_1\lambda$ to the weight of $\rho_1\lambda- \lambda$. If the length of the first row in $\lambda$ is $j$, define the \emph{critical diagonals} by
\[
\Delta_i:=\{x+y=j+3i+3/2\},
\]
and define the \emph{critical region} $R_i$ to be the region in the first quadrant between $\Delta_i$ and $\Delta_{i+1}$. In Figure \ref{fig:fig3.12}, we have depicted the critical diagonals and the region $R_0$ for a general $\lambda$.

\begin{figure}[h]
        \centering
     \begin{tikzpicture}[scale=0.55]
    % RECTANGLES
    \draw[fill=red!30] (0,1) rectangle (1,10);
    \draw[fill=green!30] (1,0) rectangle (9,1);
    
    % GRIDS
    \draw[color=gray!50] (0,0) grid (9,10);

    % NODES
    \node[] at (5,0.5) {$\lambda - \rho_1 \lambda$};
    \node[rotate=90] at (0.5,5.5) {$\rho_1 \lambda - \lambda$};
    
    \node[] at (4.5,5.5) {\small $R_0$};
    
    % HORIZONTAL DOTS
    \node[] at (6.6,5) {\large {$\dotsc$}};
    \node[] at (6,1) {\large {$\dotsc$}};
    \node[] at (3.6,5) {\large {$\dotsc$}};
    
    % VERTICAL DOTS
    \node[] at (1,5.6) {\large {$\vdots$}};
    \node[] at (3,5.6) {\large {$\vdots$}};
    \node[] at (4,4.6) {\large {$\vdots$}};
    \node[] at (6,5.6) {\large {$\vdots$}};
    \node[] at (7,4.6) {\large {$\vdots$}};
    \node[] at (9,5.6) {\large {$\vdots$}};
    
    % DIAGONALS
    \draw[color=pink] (0.5,5) -- (5.5,0);
    \draw[color=pink] (0.5,8) -- (8.5,0);
    \draw[color=pink] (1,10.5) -- (9.5,2);
    \draw[color=pink] (4,10.5) -- (9.5,5);
    \draw[color=pink] (7,10.5) -- (9.5,8);
   
    % RED HORIZONTAL SEGMENTS
    \draw[color=gray,line width=1pt] (1,4) -- (2,4) (2,3) -- (3,3) (3,2) -- (4,2) (4,1) -- (5,1) (6,2) -- (7,2) (5,3) -- (6,3) (5,9) -- (6,9) (6,8) -- (7,8) (5,6) -- (6,6); 
    \draw[color=gray,line width=1pt] (1,7) -- (2,7) (2,6) -- (3,6) (4,4) -- (5,4) (7,1) -- (8,1);
    \draw[color=gray,line width=1pt] (1,10) -- (2,10) (2,9) -- (3,9) (3,8) -- (4,8) (4,7) -- (5,7) (7,4) -- (8,4) (8,3) -- (9,3);
    \draw[color=gray,line width=1pt] (4,10) -- (5,10) (7,7) -- (8,7) (8,6) -- (9,6);
    \draw[color=gray,line width=1pt] (7,10) -- (8,10) (8,9) -- (9,9);
    
    % BLUE VERTICAL SEGMENTS
    \draw[color=gray,line width=1pt] (2,4) -- (2,3) (3,3) -- (3,2) (4,2) -- (4,1);
    \draw[color=gray,line width=1pt] (2,7) -- (2,6) (5,4) -- (5,3) (6,3) -- (6,2) (7,2) -- (7,1);
    \draw[color=gray,line width=1pt] (2,10) -- (2,9) (3,9) -- (3,8) (4,8) -- (4,7) (5,7) -- (5,6) (8,4) -- (8,3) (9,3) -- (9,2);
    \draw[color=gray,line width=1pt] (5,10) -- (5,9) (6,9) -- (6,8) (7,8) -- (7,7) (8,7) -- (8,6);
    \draw[color=gray,line width=1pt] (8,10) -- (8,9) (9,9) -- (9,8);
    
    % LABELS
    \node[] at (-0.15,1.5) {\tiny $1$};
    \node[] at (-0.15,8.5) {\tiny $j$};
    \node[] at (-0.45,9.5) {\tiny $j{+}1$};
    
    \node[] at (1.5,-0.25) {\tiny $2$};
    \node[] at (7.5,-0.25) {\tiny $j$};
    \node[] at (8.5,-0.25) {\tiny $j{+}1$};
    
    \node[] at (0.5,10.2) {\tiny $j{+}3$};
    \node[] at (1.5,10.2) {\tiny $j{+}4$};
    \node[] at (2.5,9.5) {\tiny $j{+}3$};
    \node[] at (2.3,6.5) {\tiny $j$};
    \node[] at (5.5,9.5) {\tiny $j{+}6$};
    \node[] at (4.5,10.2) {\tiny $j{+}7$};
    \node[] at (1.5,7.2) {\tiny $j{+}1$};
    \node[] at (7.5,1.2) {\tiny $j{+}1$};
    \node[] at (8.5,1.2) {\tiny $j{+}2$};

    % \node[] at (6.45,8.5) {\tiny $j+3$};
   % \node[] at (7.45,1.5) {\tiny $j+3$};
    \node[] at (9.45,0.5) {\tiny $j{+}1$};
\end{tikzpicture}
        \caption{Critical diagonals and critical region $R_0$ for general $\lambda$ and $\rho_1\lambda$. The boundary of $\lambda\cap\rho_1\lambda$ must start at $(1,j+2)$ and end at $(j+1,1)$.}
        \label{fig:fig3.12}
    \end{figure}
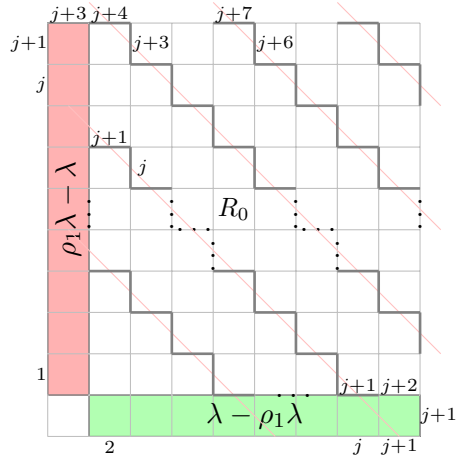

According to the labeling scheme in Figure \ref{fig:fig3.4}, all of the horizontal crossings of the critical diagonals are labeled $j+1$ mod $3$ and all of the vertical crossings of the critical diagonals are labeled $j+3$ mod $3$. Therefore, by Lemma \ref{lem:boundary}, $\w(\rho_1\lambda-\lambda)$ is the number of times the boundary of $\lambda\cap\rho_1\lambda$ crosses one of the critical diagonals vertically, and $\w(\lambda-\rho_1\lambda)$ is the number of times the boundary of $\lambda\cap\rho_1\lambda$ crosses one of the critical diagonals horizontally. Since, for any $\lambda$, the boundary of $\lambda\cap\rho_1\lambda$ begins and ends in the same critical region $R_0$, the number of vertical crossings is equal to the number of horizontal crossings, implying that $\w(\rho_1\lambda)=\w(\lambda)$.

Next, we turn to the case of $\psi_2$. As with the previous case, we overlap $\lambda$ and $\psi_2\lambda$ in the same diagram, and we need only compare $\w(\lambda-\psi_2\lambda)$ with $\w(\psi_2\lambda-\lambda)$. We define critical diagonals in this case by
\[
\Delta_i:=\{x+y=k+3i-3/2\},
\]
where $k$ is the height of the first column in $\lambda$, and we define the critical region $R_i$ to be the region between $\Delta_i$ and $\Delta_{i+1}$. See Figure \ref{fig:fig3.9}.

    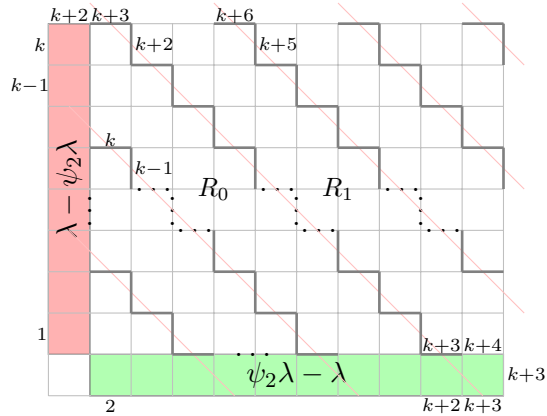
\begin{figure}[h]
        \centering
        \begin{tikzpicture}[scale=0.55]
    % RECTANGLES
    \draw[fill=red!30] (0,1) rectangle (1,9);
    \draw[fill=green!30] (1,0) rectangle (11,1);
    
    % GRID
    \draw[color=gray!50] (0,0) grid (11,9);

    % NODES
    \node[] at (6,0.5) {$\psi_2\lambda - \lambda$};
    \node[rotate=90] at (0.5,5) {$\lambda - \psi_2 \lambda$};
    
    \node[] at (4,5) {\small $R_0$};
    \node[] at (7,5) {\small $R_1$};
    
    % HORIZONTAL DOTS
    \node[] at (5.6,5) {\large {$\dotsc$}};
    \node[] at (2.6,5) {\large {$\dotsc$}};
    \node[] at (8.6,5) {\large {$\dotsc$}};
    \node[] at (6.6,4) {\large {$\dotsc$}};
    \node[] at (3.6,4) {\large {$\dotsc$}};
    \node[] at (9.6,4) {\large {$\dotsc$}};
    \node[] at (5,1) {\large {$\dotsc$}};
    
    % VERTICAL DOTS
    \node[] at (1,4.7) {\large {$\vdots$}};
    \node[] at (3,4.7) {\large {$\vdots$}};
    \node[] at (6,4.7) {\large {$\vdots$}};
    \node[] at (9,4.7) {\large {$\vdots$}};

    % DIAGONALS
    \draw[color=pink] (1,9.5) -- (10.5,0);
    \draw[color=pink] (4,9.5) -- (11.5,2);
    \draw[color=pink] (7,9.5) -- (11.5,5);
    \draw[color=pink] (10,9.5) -- (11.5,8);
    \draw[color=pink] (0.5,7) -- (7.5,0);
    \draw[color=pink] (0.5,4) -- (4.5,0);
    
    % RED HORIZONTAL SEGMENTS
    \draw[color=gray,line width=1pt] (1,9) -- (2,9) (2,8) -- (3,8) (3,7) -- (4,7) (4,6) -- (5,6) (7,3) -- (8,3) (8,2) -- (9,2) (9,1) -- (10,1) (5,8) -- (6,8) (6,7) -- (7,7) (5,2) -- (6,2) (6,1) -- (7,1);
    \draw[color=gray,line width=1pt] (4,9) -- (5,9) (7,6) -- (8,6) (10,3) -- (11,3);
    \draw[color=gray,line width=1pt] (1,6) -- (2,6)  (4,3) -- (5,3);
    \draw[color=gray,line width=1pt] (1,3) -- (2,3) (2,2) -- (3,2) (3,1) -- (4,1);
    \draw[color=gray,line width=1pt] (7,9) -- (8,9) (8,8) -- (9,8) (9,7) -- (10,7) (10,6) -- (11,6);
    \draw[color=gray,line width=1pt] (10,9) -- (11,9);
    
    % BLUE VERTICAL SEGMENTS
    \draw[color=gray,line width=1pt] (2,9) -- (2,8) (3,8) -- (3,7) (4,7) -- (4,6) (7,4) -- (7,3) (8,3) -- (8,2) (9,2) -- (9,1) (2,5) -- (2,6) (5,5) -- (5,6) (8,5) -- (8,6);
    \draw[color=gray,line width=1pt] (5,9) -- (5,8) (6,8) -- (6,7) (7,7) -- (7,6) (10,4) -- (10,3);
    \draw[color=gray,line width=1pt] (4,4) -- (4,3) (5,3) -- (5,2) (6,2) -- (6,1);
    \draw[color=gray,line width=1pt] (2,3) -- (2,2) (3,2) -- (3,1);
    \draw[color=gray,line width=1pt] (8,9) -- (8,8) (9,8) -- (9,7) (10,7) -- (10,6);
    \draw[color=gray,line width=1pt] (11,9) -- (11,8) (11,5) -- (11,6);

    % LABELS
    \node[] at (0.5,9.2) {\tiny $k{+}2$};
    \node[] at (1.5,9.2) {\tiny $k{+}3$};
    \node[] at (4.5,9.2) {\tiny $k{+}6$};
    \node[] at (2.55,8.5) {\tiny $k{+}2$};
    \node[] at (2.55,5.5) {\tiny $k{-}1$};
    \node[] at (5.55,8.5) {\tiny $k{+}5$};
    \node[] at (9.5,1.2) {\tiny $k{+}3$};
    \node[] at (10.5,1.2) {\tiny $k{+}4$};
    \node[] at (1.5,6.2) {\tiny $k$};
    \node[] at (11.55,0.5) {\tiny $k{+}3$};
    \node[] at (-0.15,1.5) {\tiny $1$};
    \node[] at (-0.45,7.5) {\tiny $k{-}1$};
    \node[] at (-0.2,8.5) {\tiny $k$};
    \node[] at (1.5,-0.25) {\tiny $2$};
    \node[] at (9.5,-0.25) {\tiny $k{+}2$};
    \node[] at (10.5,-0.25) {\tiny $k{+}3$};
\end{tikzpicture}
        \caption{Critical diagonals and critical regions $R_0$ and $R_1$ for general $\lambda$ and $\psi_2\lambda$.}
        \label{fig:fig3.9}
    \end{figure}
    
By Lemma \ref{lem:boundary}, $\w(\lambda-\psi_2\lambda)$ is the number of times the boundary of $\lambda\cap\psi_2\lambda$ crosses one of the critical diagonals vertically, and $\w(\psi_2\lambda-\lambda)$ is the number of times the boundary of $\lambda\cap\psi_2\lambda$ crosses one of the critical diagonals horizontally. Since, for any $\lambda$, the boundary of $\lambda\cap\psi_2\lambda$ begins in the critical region $R_0$ and ends in the critical region $R_1$, the number of horizontal crossings is one more than the number of vertical crossings, proving that $\w(\psi_2\lambda)=\w(\lambda)+1$.
\end{proof}

\section{Computation of highest Betti numbers}

By Proposition \ref{prop:decompose} and Theorem \ref{thm:compare}, the left-hand side of Conjecture \ref{combinatorialconjecture} can be written as a sum over admissible $\{1,2\}$-compositions, where any composition consisting of $i$ ones and $j$ twos contributes a factor of $q^{i+2j}t^j$. From this, we immediately see that the last part of Theorem \ref{maintheorem} holds: $\widetilde{b}_{k,n}=0$ for all $k>n/2$. To prove the first part of Theorem \ref{maintheorem}, we must show that the composition $(2,2,\dots,2,2)$ is admissible. To prove the second part, we must prove that the three compositions $(1,2,\dots,2,2)$, $(2,1,2,\dots,2,2)$, and $(2,2,\dots,2,1)$ are all admissible; Example \ref{example:unadmissible} then shows that these are the only three admissible compositions contributing to the coefficient of $q^{1+2j}t^j$ for $j\geq 2$. In order to prove that these compositions are all admissible, we introduce a new definition.

\begin{definition} \label{def:stair}
Let $\lambda$ denote a Young diagram with columns of height $c_i$, where $1 \leqslant i \leq j$. 
We say $\lambda$ has the \textit{stair-step property} if $c_i(\lambda) - 1 \leqslant c_{i + 1}(\lambda)$ for all $i < j$. If $\lambda$ has the stair-step property, then the \textit{landing number} of $\lambda$ is 
\[
L(\lambda) = |\{i: 1 \leqslant i < j,\ c_i(\lambda) = c_{i + 1}(\lambda)\}|.
\]
\end{definition}

The following result shows that $\psi_2$ preserves the property that $L(\lambda)\leq 2$.

\begin{lemma}\label{lem:stair}
If $\lambda$ has the stair-step property with $L(\lambda)\leq 2$, then $\psi_2\lambda$ is well-defined and has the stair-step property with $L(\psi_2\lambda)\leq 2$.
\end{lemma}

\begin{proof}
If $j$ is the number of boxes in the first row of $\lambda$ and $k$ is the number of boxes in the first column, then the assumptions imply that $k\leq j \leq k+2$. Since $j\leq k+3$, $\psi_2$ is well-defined. Since $\lambda$ satisfies the stair-step property, then $\psi_2\lambda$ will satisfy the stair-step property as long as the last step is not a big one; in other words, we need to check that the last column does not have two boxes. However, the only way that $\psi_2\lambda$ has two boxes in the last column is if $j=k+3$, which is disallowed by our assumptions on $\lambda$.

Finally, to verify the condition on the landing sets, we check case-by-case. If $L(\lambda)=0$, then $j=k$, and $\psi_2\lambda$ creates a new $3$-box landing in the final three columns, so that $L(\psi_2\lambda)=2$. If $L(\lambda)=1$, then $j=k+1$, and $\psi_2\lambda$ contains a new $2$-box landing in the final two columns, so $L(\psi_2\lambda)\in\{1,2\}$ (we take into account that $\psi_2\lambda$ could break up the landing that already existed in the first two columns of $\lambda$, see Figure \ref{fig:fig3.18}). Similarly, if $L(\lambda)=2$, then $j=k+3$ and $\psi_2\lambda$ does not create any new landings (but it might break up an old one), so $L(\psi_2\lambda)\in\{1,2\}$.
 \begin{figure}[h]
                \centering
                \begin{tikzpicture}[scale=0.5]
    % RECTANGLE
    \draw[draw=red,fill=red!30] (0,0) rectangle (1,5);
    \draw[draw=red,fill=red!30] (9,0) rectangle (14,1);
    \draw[draw=green,fill=green!30] (14,0) rectangle (16,1);
    
    % GRIDS FIRST FIGURE
    %\draw[step=1cm,color=gray!30] (0,0) grid (7,5);
    \draw[step=1cm,color=gray!30] (0,0) grid (1,5);
    \draw[step=1cm,color=gray!30] (1,0) grid (5,2);
    \draw[step=1cm,color=gray!30] (1,2) grid (3,4);
    
    %GRIDS SECOND FIGURE
    %\draw[step=1cm,color=gray!30] (9,0) grid (16,6);
    \draw[step=1cm,color=gray!30] (9,0) grid (16,1);
    \draw[step=1cm,color=gray!30] (9,1) grid (13,3);
    \draw[step=1cm,color=gray!30] (9,3) grid (11,5);
    
    % OPEN POLYGON AND LINE SEGMENT OF FIRST FIGURE
    \draw[color=gray,line width=1pt] (2,5) -- (2,4) -- (3,4) -- (3,3) -- (4,3) -- (4,2) (5,2) -- (5,1) -- (6,1) -- (6,0) -- (1,0);
    \draw[color=gray,line width=1pt] (1,0) -- (1,5) -- (2,5);
    \draw[color=gray,line width=1pt] (4,2) -- (5,2);
    
    % OPEN POLYGON AND LINE SEGMENT OF SECOND FIGURE
    \draw[color=gray,line width=1pt] (13,3) -- (13,2) -- (14,2) -- (14,1) -- (9,1) -- (9,6) -- (10,6) -- (10,5) -- (11,5) -- (11,4) -- (12,4) -- (12,3);
    \draw[color=gray,line width=1pt] (12,3) -- (13,3);
    
%    % LABELS OF FIRST FIGURE
%    \node[] at (-0.85,2.5) {\scriptsize $\lambda$ = };
%    \node[] at (0.5,-0.5) {\scriptsize $c_1$};
%    \node[] at (1.5,-0.5) {\scriptsize $c_2$};
%    \node[] at (2.5,-0.5) {\scriptsize $\dotsc$};
%    \node[] at (3.5,-0.5) {\scriptsize $\dotsc$};
%    \node[] at (4.5,-0.5) {\scriptsize $\dotsc$};
%    \node[] at (5.5,-0.5) {\scriptsize $c_j$};
%    
%    % LABELS OF SECOND FIGURE
%    \node[] at (7.9,2.5) {\scriptsize $\psi_2\lambda$ = };
%    \node[] at (9.5,-0.5) {\scriptsize $\tilde c_1$};
%    \node[] at (10.5,-0.5) {\scriptsize $\tilde c_2$};
%    \node[] at (11.5,-0.5) {\scriptsize $\dotsc$};
%    \node[] at (12.5,-0.5) {\scriptsize $\dotsc$};
%    \node[] at (13.5,-0.5) {\scriptsize $\tilde c_k$};
%    %\node[] at (14.5,-0.5) {\scriptsize $c_{k{+}1}$};
%    %\node[] at (15.5,-0.5) {\scriptsize $c_{k{+}2}$};
    
      \draw[|<-] (0,-.5) -- (1.7,-.5);
    \draw[->|] (4.3,-.5) -- (6,-.5);
    \node at (3,-.5) {\tiny $j=k+1$};

     \draw[|<-] (9,-.5) -- (11.5,-.5);
    \draw[->|] (13.5,-.5) -- (16,-.5);
    \node at (12.5,-.5) {\tiny $k+2$};
    
    % ARROW AND LABEL
    \draw[->,draw=black,line width=1pt] (6,3) -- (8.5,3);
    \node[] at (7.25,3.5) {\tiny $\psi_2$};
\end{tikzpicture}
                \caption{A new landing set when $L(\lambda)=1$. If the old landing set did not occur in the first two columns, then the new diagram would have two landings.}
                \label{fig:fig3.18}
\end{figure}
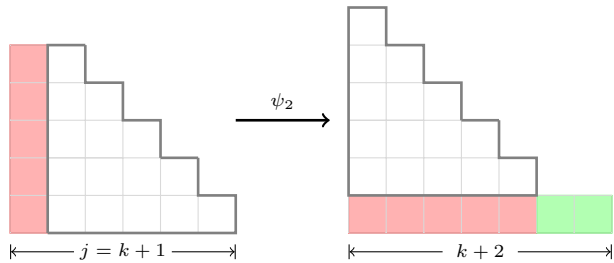

\end{proof}

With Lemma \ref{lem:stair} at our disposal, we are ready to prove the admissibility of the four compositions mentioned above.

\begin{proposition} \label{prop:admissible}
For every $n \geqslant 0$, $\psi_2^n \emptyset$, $\psi_2^n \rho_1\emptyset$, $\rho_1\psi_2^n\emptyset$, $\psi_2\rho_1\psi_2^n\emptyset$ are well-defined Young diagrams.
\end{proposition}

\begin{proof}
Since $\lambda=\emptyset$ and $\lambda=\rho_1\emptyset=\square$ both have the stair-step property with $L(\lambda)\leq 2$, we can iteratively apply Lemma \ref{lem:stair} to prove that $\psi_2^n \emptyset$ and $\psi_2^n \rho_1\emptyset$ are well-defined for all $n\geq 0$. Since $\psi_2^n\emptyset$ satisfies the stair-step property with $L(\psi_2^n\emptyset)\leq 2$, we know that $k\leq j \leq k+2$. In particular, since $j\geq k-2$, $\rho_1$ can be applied, so $\rho_1\psi_2^n\emptyset$ is well-defined for any $n\geq 0$. After applying $\rho_1$ to $\psi_2^n\emptyset$, the new first row has length $j'\leq j+1$ and the new first column has height $k'=j+1$. Since $j'\leq k'+3$, $\psi_2$ can then be applied, so $\psi_2\rho_1\psi_2^n$ is well-defined for all $n\geq 0$.
\end{proof}

\bibliographystyle{abbrv}

\begin{thebibliography}{1}

\bibitem{BB}
A.~Bia{\l}ynicki-Birula.
\newblock Some theorems on actions of algebraic groups.
\newblock {\em Ann. of Math. (2)}, 98:480--497, 1973.

\bibitem{ES}
G.~Ellingsrud and S.~A. Str\o~mme.
\newblock On the homology of the {H}ilbert scheme of points in the plane.
\newblock {\em Invent. Math.}, 87(2):343--352, 1987.

\bibitem{Gottsche}
L.~G\"{o}ttsche.
\newblock The {B}etti numbers of the {H}ilbert scheme of points on a smooth
  projective surface.
\newblock {\em Math. Ann.}, 286(1-3):193--207, 1990.

\bibitem{GZLMH}
S.~M. Gusein-Zade, I.~Luengo, and A.~Melle-Hern\'{a}ndez.
\newblock On generating series of classes of equivariant {H}ilbert schemes of
  fat points.
\newblock {\em Mosc. Math. J.}, 10(3):593--602, 662, 2010.

\bibitem{Johnson}
P.~Johnson.
\newblock Orbifold {H}ilbert schemes and generalized cores and quotients (in preparation).

\bibitem{Pak}
I.~Pak.
\newblock Partition bijections, a survey.
\newblock {\em Ramanujan J.}, 12(1):5--75, 2006.

\end{thebibliography}

\end{document}